\documentclass[11pt]{article}
\RequirePackage{amssymb,amsmath,amscd,amsthm,enumerate,geometry,pifont}
\usepackage{authblk}
\usepackage[T1]{fontenc}
\usepackage[utf8]{inputenc}

\theoremstyle{plain}

\DeclareMathOperator{\SL}{SL}
\DeclareMathOperator{\Sym}{Sym}
\DeclareMathOperator{\Alt}{Alt}
\DeclareMathOperator{\Syl}{Syl}
\DeclareMathOperator{\ps}{ps}
\DeclareMathOperator{\nps}{nps}

\renewcommand{\(}{\langle\mskip2mu\relax}
\renewcommand{\)}{\mskip2mu \rangle}

\let\normal=\trianglelefteq
\newcommand\narrow{\advance\leftskip by\parindent}

\newtheorem{thm}{Theorem}[section]

\newtheorem{cor}[thm]{Corollary}
\newtheorem{lemma}[thm]{Lemma}
\theoremstyle{definition}
\newtheorem{defn}[thm]{Definition}
\newtheorem{rem}[thm]{Remark}
\numberwithin{equation}{section}

\newcommand{\Case}[2]{\smallskip\medbreak
	\noindent\textbf{Case #1.\enspace}\emph{#2}\par
	\smallskip\noindent\ignorespaces}

\newcommand{\Claim}[1]{\medbreak
	\noindent\textbf{Claim.\enspace}\emph{#1}\par
	\smallskip\noindent\ignorespaces}

\newcommand\keywords[1]{\textbf{Keywords}: #1}

\newcommand\classification[1]{\textbf{2020 Mathematics Subject Classification}: #1}

\begin{document}
	
	\title{Groups with at most 13 nonpower subgroups}
	
	\author[a]{Jiwei Zheng}
	\author[a]{Wei Zhou\footnote{Contact zh\_great@swu.edu.cn}}
	\author[b]{D. E. Taylor}
	\affil[a]{School of Mathematics and Statistics, Southwest University, Chongqing 400715, P.R. China}
	\affil[b]{School of Mathematics and Statistics, The University of Sydney, New South Wales 2006, Australia}

	\maketitle
	
   	\noindent \textbf{Abstract.} For a group $G$ and $m\ge 1$, $G^m$ denotes the subgroup generated by
	the elements $g^m$ where $g$ runs through~$G$. The subgroups not of the
	form $G^m$ are called nonpower subgroups. We extend the classification 
	of groups with few nonpower subgroups from groups with at most 9 nonpower 
	subgroups to groups with at most 13 nonpower subgroups. \\
			
	\noindent\keywords{counting subgroups, power subgroups, nonpower subgroups.}\\
	
	\noindent\classification{20D25, 20D60.}

	\section{Introduction}
	\setcounter{secnumdepth}{1}
	For a group $G$ and $m\ge 1$ the \emph{power} subgroup $G^m$ is the subgroup
	generated by the elements $g^m$ where $g$ runs through~$G$.  A subgroup that
	is not a power subgroup is a \emph{nonpower} subgroup. Let $\ps(G)$ and $\nps(G)$
	denote the number of power and nonpower subgroups of~$G$.  It is 
	immediate that every power subgroup is a characteristic subgroup of $G$.  
	
	In 1956 Sz\'asz \cite{szasz:1956} proved that $G$ is cyclic if and only if 
	$\nps(G) = 0$. In 2006, Zhou, Shi and Duan \cite{zhou-shi-duan:2006}
	showed that a noncyclic group is finite if and only if it has a finite
	number of nonpower subgroups and it was observed that a finite noncyclic group
	must have at least three nonpower subgroups. Then Anabanti \emph{et al.} 
	\cite{anabanti-etal:2022,anabanti-hart:2022} classified the groups with
	three or four nonpower subgroups.  We summarise these results in the following 
	theorem. (See section \ref{sec:notation} for definitions of the groups.)

	\begin{thm}\label{thm:upto4}\relax\leavevmode
		For $0\le k\le 4$ a group has exactly $k$ nonpower subgroups if and only if 
		up to isomorphism it is one of the following.
	\end{thm}
	\newcounter{Lcount1}
	\begin{list}{$k=\arabic{Lcount1}$}
		{\usecounter{Lcount1}
			\setcounter{Lcount1}{-1}
			\setlength{\rightmargin}{\leftmargin}}
		\item  \quad A cyclic group.
		\item  \quad No examples.
		\item  \quad No examples.
		\item  \quad $C_2\times C_2$, $Q_8$ or\/ $G_{n,3}$ for $n\ge 1$.
		\item  \quad $C_3\times C_3$.
	\end{list}

	This theorem, combined with the following theorem from our previous paper 
	\cite{zheng-etal:2023}, classifies groups with at most nine nonpower subgroups.
	
	\begin{thm}\label{thm:upto9}\relax\leavevmode
		For $5\le k\le 9$ a group has exactly $k$ nonpower subgroups if and only if 
		up to isomorphism it is one of the following.
	\end{thm}
	
	\newcounter{Lcount2}
	
	\begin{list}{$k=\arabic{Lcount2}$}
		{\usecounter{Lcount2}
			\setcounter{Lcount2}{4}
			\setlength{\rightmargin}{\leftmargin}}
		\item\quad $C_2\times C_4$ or\/ $G_{n,5}$ for $n\ge 1$.
		\item\quad $C_5\times C_5$, $C_2\times C_2\times C_p$, $Q_8\times C_p$,
		where $p > 2$ is a prime or\/ $G_{n,3}\times C_q$ for $n\ge 1$, where $q > 3$
		is a prime.
		\item\quad $D_8$, $\Alt(4)$, $C_2\times C_8$, $Q_{16}$, $M_{4,2}$, 
		$C_3\times C_9$, $M_{3,3}$, $G_{n,7}$ or\/ $F_{n,7}$ for $n\ge 1$.
		\item\quad $C_7\times C_7$ or\/ $C_3\times C_3\times C_p$, where $p\ne 3$ is a prime.
		\item\quad $C_2\times C_{16}$, $M_{5,2}$, $C_2\times C_2\times C_{p^2}$,
		$Q_8\times C_{p^2}$, where $p > 2$ is a prime or\/ $G_{n,3}\times C_{q^2}$, 
		where $q > 3$ is a prime.
	\end{list}
	
	In the present paper we extend the classification to groups with at most 13 
	nonpower subgroups.
	
	\begin{thm}\label{thm:main}\relax\leavevmode
		For $10\le k\le 13$ a group has exactly $k$ nonpower subgroups if and only if 
		up to isomorphism it is one of the following.
	\end{thm}
	\newcounter{Lcount}
	
	\begin{list}{$k=\arabic{Lcount}$}
		{\usecounter{Lcount}
			\setcounter{Lcount}{9}
			\setlength{\rightmargin}{\leftmargin}}
		\item \quad
		$C_3\times C_{27}$, $C_2\times C_4\times C_p$,
		$M_{4,3}$,
		$\Sym(3)\times C_3$,
		$A_2 = (C_2\times C_2)\rtimes C_9$,
		$G^{(2)}_{2,n;\,5,1}$ for $n\ge 2$ or
		$G_{n,5}\times C_q$ for $n\ge 1$, 
		where $p\ne 2$ and $q\ne 2,5$ are primes.
		\item \quad
		$C_2\times C_{32}$,
		$C_5\times C_{25}$,
		$S_{16}$,
		$M_{6,2}$,
		$M_{3,5}$,
		$\SL(2,3)$,
		$G_{n,11}$ or
		$G^{(3)}_{5,n;\,11,1}$ for $n\ge 1$.
		\item \quad
		$C_4\times C_4$,
		$C_{11}\times C_{11}$,
		$Q_8\times C_{qr}$,
		$C_2\times C_2\times C_{qr}$,
		$C_3\times C_3\times C_{s^2}$,
		$C_5\times C_5\times C_p$,
		$B^2_{2,2}$,
		$G_{n,9}$ or
		$G_{n,3}\times C_{qr}$ for $n\ge 1$, where $p,q,r$ and $s$ are primes such 
		that $p\ne 5$,
		$3 < q < r$ and $s\ne 3$.
		\item \quad
		$C_2\times C_{64}$,
		$C_3\times C_{81}$,
		$M_{7,2}$,
		$M_{5,3}$,
		$\Sym(3)\times C_2 = D_{12}$,
		$C_3\rtimes Q_8$,
		$A_3 = (C_2\times C_2)\rtimes C_{27}$,
		$G_{n,13}$ or
		$F_{n,13}$ for $n\ge 1$.
	\end{list}

	\section{Notation and definitions}\label{sec:notation}
	All groups considered in this paper are finite.
	Let $\Phi(G)$ denote the Frattini subgroup of $G$. For subgroups $H$ and 
	$K$ of $G$, let $[H,K]$ be the subgroup generated by the commutators 
	$[x,y] = x^{-1}y^{-1}xy$ with $x\in H$ and $y\in K$. 
	
	
	Let $G$ be a $p$-group. For all $i\ge 1$, $\Omega_{i}(G)$ denotes the subgroup 
	$\( x\in G\mid x^{p^{i}}=1 \)$. (In general the exponent of $\Omega_{i}(G)$ may
	be greater than $p^i$.)
	
	We use Gorenstein \cite[Ch.\,5]{gorenstein:1968} as a reference for standard 
	results about $p$-groups. Statements of most of the key lemmas from 
	\cite{gorenstein:1968} used in the proof of the main theorem can be 
	found in our previous paper \cite{zheng-etal:2023}. Alternatively, see
	Huppert \cite[Ch. III]{huppert:1967}.
	
	Let $C_n$ denote the cyclic group of order $n$ and let $\Alt(n)$ and $\Sym(n)$ 
	denote the alternating and symmetric groups of a set of size $n$. Let
	$\SL(2,3)$ denote the group of $2\times 2$ matrices of determinant 1 over the
	field of 3 elements.
	
	We use the notation $G = N\rtimes K$ to mean that $G$ is a semidirect product
	of $N$ by $K$.  The action of $K$ on $N$ will be determined by the context.
	A group $G$ is \emph{metacyclic} if it has a normal cyclic subgroup $N$ such
	that $G/N$ is cyclic.
	
	Almost all the groups that occur in Theorems \ref{thm:upto4}, \ref{thm:upto9} and
	\ref{thm:main} are metacyclic or the direct product of a metacyclic group and 
	a cyclic group. Many of these groups can be described as a semidirect product 
	$C_{q^m}\rtimes C_{p^n}$ of cyclic groups of prime power orders. Metacyclic groups
	have been classified by Hempel \cite{hempel:2000}. The special case that we need
	is given by the following definition.
	
	\begin{defn}\label{defn:Gpnqm}
		For primes $p$ and $q$, positive integers $m$, $n$ and an integer $r$ 
		such that $r^{p^n}\equiv 1 \pmod{q^m}$, let $G^{(r)}_{p,n;\,q,m}$ denote the 
		group with presentation
		\[
		\( a, b \mid a^{p^n} = 1,\ b^{q^m} = 1,\ a^{-1}ba = b^r \), 
		\]
		If $p^k$ is the order of $r\pmod{q^m}$, the automorphism of $\(b\)$ 
		induced by conjugation by $a$ has order $p^k$; therefore, if $p\ne q$, then
		$p^k$ divides $q - 1$ and if the group is nonabelian, then $p$ divides $q-1$.
	\end{defn}
	
	
	\begin{rem}
		For certain values of the parameters the groups $G^{(r)}_{p,n;q,m}$ have
		standard names (see \cite{gorenstein:1968}). Where possible we prefer to use
		these names.
		\begin{enumerate}[\rm (i)]
			\item $G^{(1)}_{p,n;\,q,m}$ is the direct product $C_{p^n}\times C_{q^m}$.
			\item$G^{(-1)}_{2,1;\,q,m}$ is the \emph{dihedral group} $D_{2q^m}$ of order $2q^m$.
			\item For $n\ge 4$, $G^{(-1+2^{n-2})}_{2,1;\,2,n-1}$ is the \emph{semidihedral} 
			group $S_{2^n}$ of order $2^n$.
			\item For $n\ge 4$ when $p = 2$ and $n\ge 3$ when $p$ is an odd prime, 
			$G^{(1+p^{n-2})}_{p,1;\,p,n-1}$ is the \emph{quasidihedral} group $M_{n,p}$ 
			of order $p^n$.
			($M_{3,p}$ is the \emph{extraspecial} group of order $p^3$ and 
			exponent~$p^2$.)
		\end{enumerate}
	\end{rem}
	
	\begin{rem}
		Because of the frequency with which the groups occur in the theorems and for
		compatibility with \cite{zheng-etal:2023} we use an abbreviated notation for 
		the following special cases. For $n\ge 1$ and a prime $p$,
		\begin{enumerate}[\rm (i)]
			\item $G_{n,p^m}$ denotes $G^{(-1)}_{2,n;\,p,m}$\,;
			\item $F_{n,p}$ denotes $G^{(r)}_{3,n;\,p,1}$ where the order of $r\pmod p$ is 3
			and hence $p\equiv 1\pmod 3$;
			\item $B^2_{n,p}$ denotes $G^{(p+1)}_{p,n;\,p,2}$ and we have $B^2_{1,2} = D_8$, 
			$B^2_{2,2}$ is the semidirect product $C_4\rtimes C_4$ and for $p$ 
			odd, $B^2_{1,p}= M_{3,p}$.
		\end{enumerate}
	\end{rem}
	
	For completeness we include definitions of two other families of groups 
	with standard names and two families of groups (without standard names), 
	which occur in the proof of the main theorem.
	
	\begin{defn}\label{defn:groups}\leavevmode
		\begin{enumerate}[\rm(i)]
			\item For $n\ge 3$, $\(a,b\mid a^{2^{n-1}} = b^2 = z,\ z^2 = 1,\ b^{-1}ab = a^{-1}\)$
			is a presentation for the \emph{generalized quaternion} group $Q_{2^n}$ of order $2^n$.
			\item
			For an odd prime $p$, $\(x,y,z \mid x^p = y^p = z^p = 1,\ [x,z] = [y,z] = 1,\ [x,y] = z\)$
			is a presentation for the \emph{extraspecial} group $M(p)$ of order $p^3$ and 
			exponent~$p$.
			\item
			For $n\ge 1$, $\(a,b,c \mid a^{3^n} = b^2 = 1,\ bc = cb,\ b^a = c,\ c^a = bc \)$
			is a presentation for the group $A_n = (C_2\times C_2)\rtimes C_{3^n}$ of order 
			$2^2 3^n$. For example, $A_1\simeq\Alt(4)$.
			\item 
			For $n\ge 1$ and a prime $p$,
			\[
			\(a,b,c\mid [a,b] = c,\ a^p = b^{p^n} = c^p = 1,\ [a,c] = [b,c] = 1 \)
			\]
			is a presentation for the group $B^1_{n,p}$ of order $p^{n+2}$. Except for 
			$B^1_{1,2} = D_8$, it is non-metacyclic (see \cite[Lemma 2.5]{blackburn:1958}).
			The quotient mod $\(c\)$ is $C_p\times C_{p^n}$ and for $p$ odd, $B^1_{1,p}= M(p)$.

		\end{enumerate}
	\end{defn}

	\section{Preliminaries}\label{sec:prelim}
	First we prove a general result about cyclic subgroups of prime power order.
	
	\begin{thm}\label{thm:power}
		For all primes $p$ and finite groups $G$, if $m$ divides the exponent $e$ 
		of\/ $G$ and $G^m$ is cyclic of order $p^k$, then $m = e/p^k$. In particular, for 
		all $k$, $G$ has at most one cyclic power subgroup of order $p^k$.
	\end{thm}
	
	\begin{proof}
		Suppose that for a divisor $m$ of $e$, the power subgroup $G^m$ is cyclic 
		of order~$p^k$. Let $q\ne p$ be a prime divisor of $|G|$ and let $y$ be an element 
		of order $q^\ell$, where $q^\ell$ is the exponent of a Sylow $q$-subgroup. If $q^h$ is 
		the highest power of $q$ that divides $m$, there exist integers $s$ and $t$ such that
		$q^h = sq^\ell + tm$.  Therefore $y^{q^h} = y^{tm}\in G^m$, thus $y^{q^h} = 1$ and 
		hence $h = \ell$. This proves that $m = p^jd$ for some $j$ where $d$ is 
		the product of the exponents of the Sylow $q$-subgroups with $q\ne p$.
		Therefore, for all $g\in G$, $g^d$ is a $p$-element and it follows that
		$p^{k+j}$ is the exponent of the Sylow $p$-subgroups of $G$. Thus $m = e/p^k$.
	\end{proof}
	
	
	\begin{cor}\label{cor:power}
		Let $p^f$ be the largest order of a $p$-element of\/ $G$, where $p$ is a prime 
		divisor of\/ $|G|$. Let $p^k$ be the largest order of a cyclic power $p$-subgroup 
		of\/ $G$. If $p$ is odd and a Sylow $p$-subgroup is not cyclic, then $G$ contains at 
		least $pf - k + 1$ cyclic nonpower $p$-subgroups.
	\end{cor}
	
	\begin{proof}
		Let $P$ be a Sylow $p$-subgroup of $G$ and suppose that $P$ is not cyclic. By a 
		Theorem of Burnside \cite[\S105]{burnside:1955}, for all $m$ where $1\le m\le f$,
		$P$ contains more than one subgroup of order $p^m$. Miller\cite{miller:1929} has
		proved that for $m > 1$ the number of cyclic subgroups of order $p^m$ is divisible
		by $p$ and the number of subgroups of order $p$ is congruent to $1+p$ modulo $p^2$
		(see also Berkovich \cite[Th.\ 1.10]{berkovich:2008}). Therefore $P$ has at least 
		$p$ subgroups of order $p^i$ for $2\le i\le f$ and at least $p+1$ subgroups of
		order $p$. Furthermore, from Theorem \ref{thm:power} $P$ has exactly one cyclic 
		power subgroup of order $p^i$ for $1\le i\le k$. Therefore $G$ has at least 
		$p + (k-1)(p-1) + p(f-k) = pf-k+1$ cyclic nonpower $p$-subgroups.
	\end{proof}
	
	The following two lemmas play an essential r\^ole in the proof of
	the main theorem.
	
	\begin{lemma}[{\cite[Lemma 3]{anabanti-etal:2022}}]\label{lemma:prod}
		If $A$ and $B$ are finite groups such that $|A|$ and $|B|$ are coprime, then
		$\ps(A\times B) = \ps(A)\ps(B)$ and 
		$\nps(A\times B) = \nps(A)s(B) + \ps(A)\nps(B)$, where $s(B)$ is the number of
		subgroups of $B$.
	\end{lemma}
	
	\begin{lemma}[{\cite[Lemma 2]{zhou-shi-duan:2006}}]\label{lemma:quo}
		Suppose $N$ and $H$ are subgroups of $G$ such that $N\normal G$. If $HN/N$ is a 
		nonpower subgroup of $G/N$, then $H$ is a nonpower subgroup of $G$. Therefore 
		$\nps(G)\ge \nps(G/N)$.
	\end{lemma}
	
	\begin{rem}
		This lemma was stated in \cite{zheng-etal:2023} with the additional condition
		$N\subseteq H$. However, the same proof shows that it remains true without
		this restriction. It often can be used to show that $\nps(G)\ge 2\nps(G/N)$. For 
		example, suppose that $|N|$ is coprime to $|G/N|$.  From the Schur--Zassenhaus 
		Theorem \cite[Th.\ 6.2.1]{gorenstein:1968}, $G$ is a semidirect product 
		$N\rtimes K$. In this case, if $H\subseteq K$ and if $HN/N$ is a nonpower subgroup
		of $G/N$, then both $H$ and $HN$ are nonpower subgroups of $G$. Since $H\ne HN$
		it follows that $\nps(G)\ge 2\nps(G/N)$.
	\end{rem}
	
	\begin{lemma}[\cite{zheng-etal:2023}, Theorem 2.12]\label{lemma:qd}
		There is no finite $p$-group $G$ such that $G/N \simeq M_{n,p}$, where $N$ is a 
		central subgroup of $G$ of order $p$ contained in $G'$.
	\end{lemma}
	
	The following lemma from \cite{zheng-etal:2023} is a consequence of Lemma 2.2 and 
	Theorem 2.3 of Blackburn \cite{blackburn:1958}.
	\begin{lemma} \label{lemma:R}
		For a non-abelian $p$-group $G$ generated by two elements, let $R=\Phi(G')G_{3}$ 
		where $G_{3}=[[G,G],G]$. Then 
		\begin{enumerate}[\rm(i)]
			\item $R$ is the only maximal subgroup of $G'$ that is normal in $G$,
			\item $G$ is metacyclic if and only if $G/R$ is metacyclic, 
			\item If the type of $G/G'$ is $(p,p^{n})$ and $G/R$ has no cyclic maximal subgroup, then $G/R$ is isomorphic to $B_{n,p}^{1}$ or $B_{n,p}^{2}$.
		\end{enumerate}
	\end{lemma}

	\section{Nonpower values}\label{sec:values}
	Section 3 of \cite{zheng-etal:2023} provides formulas for the number of nonpower 
	subgroups for many families of groups. For convenience, we summarise this 
	information in the following lemma, then prove formulas for the additional values 
	of $\nps(G)$ needed in the proof of the main result.  The formulas for 
	$\nps(M_{n,p})$, $\nps(G_{n,p^k})$ and $\nps(F_{n,p})$ are special cases
	of Lemmas \ref{lemma:Gpp} and \ref{lemma:Gnp} below.
	
	\begin{lemma}\label{lemma:list}
		For an integer $n$ and a prime $p$ we have:
		\begin{enumerate}[\rm (i)]
			\item for $n\ge 3$, $\nps(D_{2^n}) = 2^n - 1$;
			\item for $n\ge 3$, $\nps(Q_{2^n}) = 2^{n-1} - 1$;
			\item for $n\ge 4$, $\nps(S_{2^n}) = 3\cdot 2^{n-2} - 1$;
			\item for $n\ge 3$, $\nps(M_{n,p}) = p(n-1) + 1$ (when $p = 2$, assume $n\ge 4$);
			\label{i:QD}
			\item $\nps(M(p)) = p^2 + 2p + 2$;\label{i:esp}
			\item if $p > 2$, then $\nps(G_{n,p^k}) = p(p^k-1)/(p-1)$;
			\item if $p\equiv 1\pmod 3$, then $\nps(F_{n,p}) = p$;
			\item for $n\ge 1$, $\nps(A_n) = 3n + 4$;\label{i:extA4}
			\item \label{i:Cp}
			{\rm \cite[Th. 3.3]{Marius:2010}}
			for $n_{2}\ge n_{1}\ge 1$ and a prime $p$ the value of 
			$\nps(C_{p^{n_{1}}}\times C_{p^{n_{2}}})$ is
			\[
			\hspace{-4pt}\frac{(n_{2}-n_{1}+1)p^{n_{1}+2}-(n_{2}-n_{1}-1)p^{n_{1}+1}-(n_2+1)p^2+(n_{2}-n_{1}+1)p+n_{2}}{(p-1)^{2}}.
			\]
		\end{enumerate}
	\end{lemma}
	
	\begin{lemma}\label{lemma:Gpp}
		For prime $p>2$, $\nps(G^{(r)}_{p,n;p,m}) = \nps(C_{p^n}\times C_{p^m})$.
	\end{lemma}
	\begin{proof}
		Let $G = G^{(r)}_{p,n;\,p,m}$. Since $G'$ is cyclic, we know from 
		\cite[III \S10]{huppert:1967} that $G$ is a regular $p$-group. Then from 
		\cite[Th.\ 4.21]{hall:1934}, $G^{p^s}$ is the set $\{\,g^{p^s}\mid g\in G\,\}$. 
		Thus there are exactly $\max(n,m)+1$ power subgroups in $G$. This means that 
		$\ps(G)=\ps(C_{p^n}\times C_{p^m})$. It follows from 
		\cite[Prop.\ 1]{mann:2010} that $G$ is lattice-isomorphic to 
		$C_{p^n}\times C_{p^m}$. Thus $s(G) = s(C_{p^n}\times C_{p^m})$ and therefore 
		$\nps(G) = \nps(C_{p^n}\times C_{p^m})$.
	\end{proof}
	
	From (i) and (ix) of Lemma \ref{lemma:list} we know that for $p=2$ the conclusion of the Lemma \ref{lemma:Gpp} is not valid.
	
	\begin{lemma}\label{lemma:Gnp}
		Suppose that $p\ne q$ are primes and that $m$, $n$ and $r$ are positive integers such 
		that $r\ne 1$ and $r^{p^n}\equiv 1\pmod{q^m}$. Let $G = G^{(r)}_{p,n;q,m}$ and let 
		$p^k$ be the order of $r\pmod{q^m}$. Then $p\mid q - 1$ and $\nps(G) = kq(q^m-1)/(q-1)$.
	\end{lemma}
	\begin{proof}
		For $0\le i < k$ and $0\le j < m$ let $P_i = \(a^{p^i}\)$, $Q_j = \(b^{q^{m-j}}\)$
		and $H_{ij} = P_iQ_j$. From the presentation of $G$ we have $Q_j\normal G$ and
		$Z(G)=\(a^{p^k}\)$. Furthermore, $P_0\in\Syl_p(G)$ and $G$ is the semidirect product 
		$Q\rtimes P_0$, where $Q\in\Syl_q(G)$. It is clear 
		that $P_0\subseteq N_G(P_i)$ and we claim that for $0\le i< k$, $N_G(P_i) = P_0$. If
		$x\in Q\cap N_G(P_i)$, then $[P_i,\(x\)]\subseteq P_i\cap Q = 1$; that is,
		$x\in C_G(P_i)$. It follows from \cite[Th. 5.2.4]{gorenstein:1968} that $x\in C_G(P_0)$ 
		and hence $x\in Z(G)\cap Q = 1$.  This establishes the claim.
		
		We also have $P_0\subseteq N_G(H_{ij})$ and it follows from the Frattini argument 
		that $N_G(H_{ij}) = Q_jN_G(P_i) = Q_j\rtimes P_0$. Therefore $H_{ij}$ has $q^{m-j}$ 
		conjugates in $G$ and since for $0\le i < k$ and $0\le j < m$ every subgroup of order 
		$p^{n-i}q^j$ is conjugate to $H_{ij}$ we have found 
		$k\sum_{j=0}^{m-1} q^{m-j} = kq(q^m-1)/(q-1)$ subgroups, all of which are nonpower 
		subgroups.  For $1\le i\le n$ the power subgroup $G^{p^{n-i}}$ is the unique subgroup 
		of order $p^iq^m$. For $k\le s\le n$ and $0\le j < m$, the power subgroup 
		$G^{p^sq^{m-j}}$ is the unique subgroup of order $p^{n-s}q^j$.  This accounts for 
		all subgroups of $G$. Therefore $\nps(G) = kq(q^m-1)/(q-1)$.
	\end{proof}
	

	\begin{lemma}\label{lemma:B }
		For $n\ge 2$, we have
		\begin{enumerate}[\rm(i)]
			\setlength{\itemindent}{-0.5em}
			\item  $\nps(B^1_{n,p}) = p^2(2n-1) + p(n+1) + 2$, and
			\item  $\nps(B^2_{n,p}) = p^2(n-1) + p(n+1) + 2$.
		\end{enumerate}
	\end{lemma}
	
	\begin{proof}
		There are $n+1$ power subgroups of $B_{n,p}^{1}$ and $B_{n,p}^{2}$ . We count all the subgroups of 
		them by considering their exponents. 
		
		Notice that $\Omega_{n-1}(B^1_{n,p}) =
		\(a\)\times\(c\)\times\(b^p\)\simeq C_p\times C_p\times C_{p^{n-1}}$.  By Corollary 2.2 
		in \cite{tarnauceanu-toth:2017}, the number of subgroups of exponent at most $p^{n-1}$ 
		is $s(C_p\times C_p\times C_{p^{n-1}}) = 2(n-1)p^2+np+n+2$. Furthermore, there are 
		$p^2+p+1$ subgroups of exponent $p^n$. Thus $\nps(B^1_{n,p}) = p^2(2n-1) + p(n+1) + 2$. 
		
		Similar to $B^1_{n,p}$, $\Omega_{n-1}(B^2_{n,p}) =
		\(a\)\times\(b^p\)\simeq C_{p^{2}}\times C_{p^{n-1}}$. From Corollary 2.2 
		in \cite{tarnauceanu-toth:2017}, the number of subgroups of exponent at most $p^{n-1}$ 
		is $s(C_{p^2}\times C_{p^{n-1}}) = (n-2)p^2+pn+n+2$. And there are also $p^2+p+1$ subgroups of exponent $p^n$. Thus $\nps(B^2_{n,p}) = p^2(n-1) + p(n+1) + 2$.

	\end{proof}
	
	\begin{lemma}\label{lemma:Gn3C3}
		Let $G = G_{m,3}\times \underbrace{C_3\times\dots\times C_3}_n$. For
		$n \ge 1$ we have $\nps(G)\ge 4m+6$ and equality holds when $n = 1$.
	\end{lemma}
	\begin{proof}
		Similar to the calculation of $\nps(G_{n,3})$ in \cite{zheng-etal:2023}, we have 
		$\nps(G_{m,3}\times C_3) = 4m+6$. And as $n$ increases, the number of subgroups 
		will increase but the number of power subgroups does not increase. Thus 
		$\nps(G)\ge 4m+6$. This completes the proof.
	\end{proof}
	
	\begin{lemma}\label{lemma:hamilton}
		Let $G = Q_8\times \underbrace{C_2\times\dots\times C_2}_n$. For
		$n \ge 1$ we have $\nps(G)\ge 16$ and equality holds when $n = 1$.
	\end{lemma}
	\begin{proof}
		For all $n\ge 1$ the exponent of $G$ is $4$ and the only proper non-trivial power 
		subgroup is $G^2$ of order $2$. The group $Q_8\times C_2$ has $19$ subgroups, 
		consequently $\nps(Q_8\times C_2) = 16$ and $\nps(G)\ge 16$ for $n\ge 1$.
	\end{proof}
	
	\begin{lemma}\label{lemma:S4}
		By direct calculation we have $\nps(\Sym(4)) = 26$, $\nps(\SL(2,3)) = 11$ and 
		$\nps(C_3\rtimes Q_8) = 13$.
	\end{lemma}
	
	\begin{lemma}[{\cite[Lemma 3.3]{zheng-etal:2023}}]\label{lemma:dc2}
		Let $G = D_{2p}\times \underbrace{C_2\times\dots\times C_2}_n$. For
		$n \ge 1$ and a prime $p > 2$, we have $\nps(G)\ge 3p+4$ and equality holds 
		when $n = 1$.
	\end{lemma}
	
	\begin{lemma}[{\cite[Lemma 3.4]{zheng-etal:2023}}]\label{lemma:dc3}
		Let $X_{n,p} = D_{2p}\times \underbrace{C_3\times\dots\times C_3}_n$.
		For $n \ge 1$ and a prime $p > 3$, we have $\nps(X_{n,p})=(p+3)s(C_{3}^{n})-6\geq10$ 
		where $C_3^n = \underbrace{C_3\times\dots\times C_3}_n$. 
		For $n > 2$, $\nps(X_{n,3}) > \nps(X_{2,3}) = 48$ and $\nps(X_{1,3}) = 10$.
	\end{lemma}
	
	\section{Proof of Theorem \ref{thm:main}}
	
	\begin{lemma}\label{lemma:A}
		Let $P$ be a Sylow $p$-subgroup of a group $G$ such that $P\ne N_G(P)\ne G$. 
		Then $10\le \nps(G)\le 13$ if and only if for some $n\ge 1$ one of the following holds
		\begin{enumerate}[\quad\rm 1)]
			\item $\nps(G) = 10$, $p = 2$ and $G$ is isomorphic to $\Sym(3)\times C_3$ or
			to $G_{n,5}\times C_r$ where $r\ne 2,5$ is a prime;
			\item $\nps(G) = 11$, $p = 3$ and $G$ is isomorphic to $\SL(2,3)$;
			\item $\nps(G) = 12$, $p = 2$ and $G$ is isomorphic to $G_{n,3}\times C_{qr}$ where 
			$q$ and $r$ are primes such that $3 < q < r$.
		\end{enumerate}
		In all cases, if\/ $Q\in\Syl_q(G)$ and $q\ne p$, then $Q\normal G$.
	\end{lemma}
	
	\begin{proof}
		Since $P\ne N_G(P)\ne G$, both $P$ and $N_G(P)$ have at least $p+1$ conjugates
		and so $2(p+1)\le\nps(G)\le 13$, hence $p$ is 2, 3 or 5. 
		
		First we show that $p\ne 5$. Suppose to the contrary that $p = 5$. Then 
		the 6 conjugates of $P$ and the 6 conjugates of $N_G(P)$ are nonpower subgroups,
		hence every other subgroup is normal. In particular, if $Q\in\Syl_3(G)$, then
		$Q\normal G$. The permutation action of $G$ on $\Syl_5(G)$ defines a homomorphism
		$G \to \Sym(6)$ with kernel $K$. Then $|KP/K| = 5$ and $KQ/K$ is a normal 3-subgroup 
		of $G/K$. But $\Sym(6)$ does not contain a 3-subgroup normalised by a group of 
		order 5, therefore $p \ne 5$.
		
		\Claim{If $Q\in\Syl_q(G)$ and $q\ne p$, then $Q\normal G$.}
		If $N_G(Q)\ne G$, then $Q = N_G(Q)$ otherwise $G$ would have least 
		$2(p+1)+2(q+1) \ge 14$ nonpower subgroups. 
		
		There are at least $p+1$ conjugates of $P$ and $N_G(P)$ and at least $q+1$
		conjugates of $Q$, therefore $2(p+1) + (q+1)\le 13$ and so $q\le 5$. If $q = 5$, 
		then $p = 2$ and $|G:Q| = 6$. But then $|P| = 2$ and $G$ would have a normal 
		subgroup $H$ of index 2, which must contain the 6 conjugates of $Q$. But
		$|H:Q| = 3$, which is a contradiction and therefore $q\ne 5$.  It follows
		that $q\le 3$, $G = PQ$ and $|P| = |G:Q|$.
		
		Suppose that $p = 2$. Then $q = 3$, $|G : Q| = 4 = |P|$ and the permutation action 
		of $G$ on $\Syl_3(G)$ defines a homomorphism $G\to\Sym(4)$ with kernel $K\subseteq Q$. 
		It follows that $G/K\simeq \Alt(4)$ and therefore $P\simeq C_2\times C_2$. But 
		$N_G(P)\ne G$, therefore $K\ne 1$ and so for each element $x\in P$ of order 2, it 
		follows from Lemma \ref{lemma:quo} that $\(x\)$ and $\(x\)K$ are nonpower subgroups. 
		Taking into account the 3 conjugates of $P$, the 3 conjugates of $N_G(P)$ and the 
		4 conjugates of $Q$, $G$ would have at least 16 nonpower subgroups, contrary to
		our assumption.
		
		This leaves the possibility that $p = 3$, $q = 4$ and hence $|G:Q| = 3 = |P|$. 
		Again the permutation action of $G$ on $\Syl_3(G)$ defines a homomorphism 
		$G\to\Sym(4)$ this time with kernel $K\subseteq N_G(P)$. But $P\nsubseteq K$, 
		therefore $K\subseteq Q$ and Lemma \ref{lemma:S4} implies $G/K\not\simeq\Sym(4)$.
		Consequently $G/K\simeq \Alt(4)$ and hence $Q\normal G$.\qed
		
		\medskip
		If $p = 2$, then $| G : N_G(P)| = 2k + 1$ for some $k$ and so $2(2k+1)\le 13$.
		Therefore $q = |G : N_G(P)|$ is either 3 or 5 and we choose $Q\in\Syl_q(G)$.
		Similarly, if $p = 3$, then $|G : N_G(P)| = 4$. In this case we let $q = 2$
		and choose $Q\in\Syl_2(G)$. In all cases $G = N_G(P)Q$.
		
		\Claim{$G = PQ\times A$, where $A$ is a cyclic group of order 1, $r$, $r^2$ or
			$rs$, where $r$ and $s$ are primes different from $p$ and $q$.}
		Let $R$ be a Sylow $r$-subgroups of $G$, where $r\ne p,q$ is a prime. 
		Then $R$ acts by conjugation on $\Syl_p(G)$ and in all cases $R$ fixes a conjugate 
		of $P$. Thus $R\subseteq N_G(P)$ and since $R\normal G$ we have $[P,R]\subseteq P\cap R = 1$.
		We also have $Q\normal G$ and so $[Q,R]\subseteq Q\cap R = 1$. Let $A$ 
		be the product of the Sylow $r$-subgroups for $r\ne p,q$. Then $G = PQ\times A$
		and from Lemma \ref{lemma:prod}
		\[
		\nps(G) = \nps(PQ)s(A) + \ps(PQ)\nps(A).
		\]
		
		We have $\nps(PQ)\ge 3$ and $\ps(PQ)\ge 3$. Thus $A$ is cyclic, otherwise 
		$s(A)\ge 3$ and from Theorem \ref{thm:upto4} $\nps(A)\ge 3$ hence $\nps(G) > 13$.
		Since $A$ is cyclic we have $\nps(A) = 0$ 
		and therefore $\nps(G) = \nps(PQ)s(A)\le 13$. Thus $s(A)\le 4$ and $|A|$ is 1,
		$r$ or $rs$, where $r\le s$ are primes.\qed
		
		\Claim{If $p = 2$, then $G$ is isomorphic to $\Sym(3)\times C_3$, $G_{n,5}\times C_r$
			where $r\ne 2,5$ is a prime or $G_{n,3}\times C_{rs}$ where $r$ and $s$ are primes 
			such that $3 < r < s$.}
		We have $G = PQ\times A$ where $A$ is cyclic and $\nps(G) = \nps(PQ)s(A)$.
		If $A\ne 1$, then $s(A)\ge 2$ and therefore $\nps(G)\le 6$. It follows from Theorems 
		\ref{thm:upto4} and \ref{thm:upto9} that for some $n$, $G\simeq G_{n,3}$ or $G_{n,5}$.
		If $s(A) = 2$, then $A\simeq C_r$; if $s(A) = 3$, then $A\simeq C_{r^2}$; if
		$s(A) = 4$, then $A\simeq C_{rs}$, for some primes $r$ and $s$.
		
		\smallskip
		This leaves the possibility that $A = 1$ and $G = PQ$. If $|G:N_G(P)| = 5$, 
		then $P$ and $N_G(P)$ each have 5 conjugates and if $Q$ is not cyclic it follows
		from Theorem \ref{thm:power} that $Q$ contains at least 5 nonpower subgroups of
		order 5.  This contradicts the assumption that $\nps(G)\le 13$ and therefore
		$|G:N_G(P)| = 3$.
		
		We have $[N_Q(P),P] \subseteq P\cap Q = 1$ and therefore $N_Q(P) = C_Q(P)$. 
		Consequently $N_G(P) = P\times C_Q(P)$. Let $K$ be the kernel of the 
		permutation action on $\Syl_2(G)$. Then $G/K\simeq\Sym(3)$ and $K\subseteq N_G(P)$, 
		hence $K\cap Q = C_Q(P)$ and $|Q:C_Q(P))| = 3$.
		
		If $P$ is not cyclic, then $\nps(P)\ge 3$ and for each nonpower subgroup $H$ of
		$P$ it follows from Lemma \ref{lemma:quo} that $H$, $HQ$ and $HC_Q(P)$ are nonpower 
		subgroups of $G$ hence $\nps(G)\ge 15$, which contradicts the assumption 
		$\nps(G)\le 13$. Therefore $P$ is cyclic and so $N_G(P) = C_G(P)$.
		
		Let $P = \(a\)$. Then $P\cap K = \(a^2\)$ and $K = \(a^2\)\times C_Q(P)$. Thus 
		$\(a^2\)$ is a characteristic subgroup of $K$ and therefore $\(a^2\)\normal G$. It 
		follows that $z = a^2\in Z(G)$. Let $b$ be a conjugate of $a$ such that $b\ne a$
		and let $x = ab^{-1}$. Then $b^2 = z$, $b^{-1} = bz^{-2}$ and $a^{-1}xa = x^{-1}$. 
		Since $x\in \(z\)Q$ and $P$ is cyclic, $y = x^{2^k}\ne 1$ belongs to $Q$ for some~$k$.
		Furthermore $a$ has only 3 conjugates therefore $y^3 = 1$. The group $\(y\)P$ 
		is generated by the conjugates of $P$ and so $\(y\)P\normal G$. Moreover, for some 
		$n$ we have $\(y\)P\simeq G_{n,3}$.
		
		We now have $G = \(y\)P\times C_Q(P)$. If $C_Q(P)$ is not cyclic, it contains
		an elementary abelian subgroup $E$ of order 9. Then $\(y\)\times H$ is 
		elementary abelian of order 27; it contains 13 subgroups of order 3, at most one 
		of which can be a nonpower subgroup (Theorem \ref{thm:power}). It follows
		that $C_Q(P)$ is cyclic. If the exponent of $C_Q(P)$ is $3^e$, then $Q^3$ is
		a power subgroup of order $3^{e-1}$ and from Corollary \ref{cor:power},
		$Q$ contains at least $2e + 2$ nonpower subgroups of $G$.  If $e\ge 2$, then
		$P$, $P\times C_Q(P)$ and $P\times Q^3$ each have 3 conjugates and thus 
		$\nps(G)\ge 15$, which contradicts our assumption.  Therefore $G\simeq
		G_{n,3}\times C_3$ and from Lemma \ref{lemma:Gn3C3} we have $n = 1$;
		that is, $G\simeq\Sym(3)\times C_3$.
		
		\Claim{If $p = 3$, then $\nps(G) = 11$ and $G\simeq\SL(2,3)$.}
		In this case $|G : N_G(P)| = 4$ and the permutation action on $\Syl_3(G)$ defines 
		a homomorphism $G\to\Sym(4)$ whose kernel $K$ is a proper subgroup of $N_G(P)$. 
		Since $Q\normal G$ we must have $G/K\simeq \Alt(4)$ and $QK/K\simeq C_2\times C_2$.
		The assumption $10\le\nps(G)$ implies $K \ne 1$. Let $x_1, x_2, x_3\in Q$ 
		be elements such that $x_iK$ is an involution in $QK/K$. From Lemma \ref{lemma:quo},
		$\(x_i\)$ and $\(x_i\)K$ are nonpower subgroups. The 4 conjugates of $P$, the 4 
		conjugates of $N_G(P)$ are nonpower subgroups and since there are at most 13 
		nonpower subgroups in $G$, it follows that $\(x_i\) = \(x_i\)K$. That is, 
		$K\subseteq \(x_i\)$ and so $K = \(x_i^2\)$ for $1\le i\le 3$. We 
		now see that $|P| = 3$, $K\subseteq C_G(P)$ and $P$ permutes the $x_i$, hence
		$x_1^2 = x_2^2 = x_3^2$. It follows that $Q\simeq Q_8$
		and thus $G\simeq \SL(2,3)$.
	\end{proof}
	
	\begin{lemma}\label{lemma:B}
		Let $P$ be a Sylow $p$-subgroup of $G$ such that  $P = N_G(P)\ne G$. 
		Then $10\le\nps(G)\le 13$ if and only if for some $n\ge 1$
		one of the following holds
		\begin{enumerate}[\quad\rm 1)]
			\item $\nps(G) = 10$, $p = 2$ and $G$ is isomorphic to $G^{(2)}_{n,5}$ for $n\ge 2$.
			\item $\nps(G) = 10$, $p = 3$ and $G$ is isomorphic to $A_2 = (C_2\times C_2)\rtimes C_9$.
			\item $\nps(G) = 11$, $p = 2$ and $G$ is isomorphic to $G_{n,11}$.
			\item $\nps(G) = 11$, $p = 5$ and $G$ is isomorphic to $G^{(3)}_{5,n;\,11,1}$.
			\item $\nps(G) = 12$, $p = 2$ and $G$ is isomorphic to $G_{n,9}$.
			\item $\nps(G) = 13$, $p = 2$ and $G$ is isomorphic to $G_{n,13}$,
			$\Sym(3)\times C_2 = D_{12}$ or $C_3\rtimes Q_8$.
			\item[$7)$] $\nps(G) = 13$, $p = 3$ and $G$ is isomorphic to $F_{n,13}$ or
			$A_3 = (C_2\times C_2)\rtimes C_{27}$.
		\end{enumerate}
	\end{lemma}
	
	\begin{proof}
		The Sylow subgroup $P$ has $m = |G:N_G(P)|$ conjugates. Since $m\equiv 1\pmod p$
		and $N_G(P)\ne G$, we have $m \ge p+1$ and the conjugates of $P$ are nonpower subgroups. 
		The assumption $\nps(G)\le 13$ implies $p\in\{2,3,5,7,11\}$. 
		
		If $p = 2$, then $m\in\{3,5,7,9,11,13\}$; if $p = 3$, then $m\in\{4,7,10,13\}$; if 
		$p = 5$, then $m\in \{6,11\}$; if $p = 7$, then $m = 8$; if $p = 11$, then $m = 12$. 
		
		Suppose that $p = 3$ and $m = 10$ or $p = 5$ and $m = 6$. Then $G$ is a group of 
		twice odd order and therefore has a subgroup $H$ of index 2, which contains $P$. 
		Then $|H:P| = m/2$, which is impossible. Suppose that $p = 11$ and $m = 12$. The 
		permutation action on $\Syl_{11}(G)$ defines a homomorphism $G\to\Sym(12)$ whose 
		kernel $K = \bigcap_{g\in G}P^g$ is properly contained in $P$. Let $\overline G = G/K$.
		Then $N_{\overline G}(P/K) = C_{\overline G}(P/K)$ and therefore by Burnside's normal 
		$p$-complement theorem (see \cite[Th.\ 7.4.3]{gorenstein:1968}) $P/K$ has a normal 
		11-complement in $\overline G$, which is a contradiction.
		
		Thus in all cases $m$ is a power of a prime $q$ and for $Q\in\Syl_q(G)$ we have 
		$G = PQ$ and $|Q| = m$.  We shall show that $Q\normal G$.
		
		If $N_G(Q)\ne G$, it follows from Lemma \ref{lemma:A} that $Q = N_G(Q)$, otherwise
		$P\normal G$. If the order of $Q$ is $q$ or $q^2$, then $Q$ is abelian and it follows 
		from Burnside's normal $p$-complement theorem that $P\normal G$. Thus the only
		possibility is $p = 7$, $m = 8$ and $q = 2$.  But then $|G : Q|$ must be a power
		of 7 and so $\nps(G)\ge 15$, contrary to our assumption.  This proves that
		$Q\normal G$.
		
		The permutation action on $\Syl_p(G)$ defines a homomorphism $G\to\Sym(m)$ whose 
		kernel $K = \bigcap_{g\in G}P^g$ is properly contained in $P$. Since $Q\normal G$
		we have $[K,Q]\subseteq K\cap Q = 1$. Moreover $Q$ acts transitively on the $m$ 
		conjugates of $P$ and therefore $K = C_P(Q)$.
		
		\Case{1}{$p = 2$ and $m\in\{3,5,7,9,11,13\}$.}
		We have $G = Q\rtimes P$, where $P$ is a 2-group and $Q$ is an abelian
		group of order $m$. We treat each value of $m$ separately.
		
		\Case{1a}{$p = 2$ and $|Q| = 3$.}
		{\narrow
			In this case $G/K\simeq\Sym(3)$ and $|P/K| = 2$. Therefore 
			$\Phi(P)\subseteq K = C_P(Q)$ and so $\Phi(P)\normal G$. Thus 
			\[
			G/\Phi(P)\simeq\Sym(3)\times\underbrace{C_2\times\dots\times C_2}_n.
			\]
			From Lemma \ref{lemma:dc2} $n \le 1$. If $n = 0$, then $P$ is cyclic,
			$G\simeq G_{n,3}$ and $\nps(G) = 3$, which contradicts the assumption 
			$10\le \nps(G)$.  Thus $n = 1$, $G/\Phi(P)\simeq\Sym(3)\times C_2\simeq D_{12}$
			and from Lemma \ref{lemma:dc2} $\nps(G/\Phi(P)) = 13$. The only non-trivial 
			proper power subgroup of $\Sym(3)\times C_2$ is the subgroup of order 3. Therefore, 
			for all $H\subseteq P$ such that $|H\Phi(P)/\Phi(P)| = 2$ it follows from 
			Lemma \ref{lemma:quo} that both $H$ and $H\Phi(P)$ are nonpower subgroups.
			Therefore $H = H\Phi(P)$ and hence $\Phi(P)\subseteq H$. 
			
			If $P\setminus \Phi(P)$ contains and element of order 2, then $\Phi(P) = 1$
			and thus $G\simeq\Sym(3)\times C_2$. Otherwise $\Phi(P) = \(x^2\)$ for all 
			$x\in P\setminus\Phi(P$.  Thus $P$ is nonabelian and it follows from Lemma 
			\ref{lemma:quo} that $\nps(G)\ge 2\nps(G/Q)$.  Then Theorems \ref{thm:upto4} 
			and \ref{thm:upto9} show that the only possibility for $P\simeq G/Q$ is the 
			quaternion group $Q_8$.
			Thus $G$ is the semidirect product
			\[
			C_3\rtimes Q_8 = \( x,y,b\mid x^4 = y^4 = b^3 = [y,b] = 1,\,
			x^2 = y^2,\,[x,y] = x^2,\,b^x = b^{-1}\).
			\]}
		
		\Case{1b}{$p = 2$ and $|Q| = 5$.}
		{\narrow
			In this case $QK/K$ is a normal subgroup of $G/K\subseteq \Sym(5)$. Therefore 
			$|G/K|$ is either 10 or 20. 
			
			Suppose $|G/K| = 10$. If $P$ is not cyclic, then $\Phi(P)\subseteq K$ and therefore 
			$\Phi(P)\normal G$. We have $G/\Phi(P)\simeq D_{10}\times C_2\times\dots\times C_2$ 
			and Lemmas \ref{lemma:quo} and \ref{lemma:dc2} imply $\nps(G)\ge 19$ whereas we 
			assume that $\nps(G)\le 13$, hence $P$ is cyclic. If $a$ generates $P$ and $b$ 
			generates $Q$, then $b^a = b^{-1}$ and $G\simeq G_{n,5}$. But then $\nps(G) = 5$
			and so $G$ doesn't satisfy the assumption $10\le\nps(G)$.
			
			Thus $|G/K| = 20$, $G/K\simeq G^{(2)}_{2,n;\,5,1}$ and $\nps(G/K) = 10$. 
			Choose $a\in P$ such that $aK$ generates $P/K$. The five conjugates of $\(a\)$ and 
			the five conjugates of $\(a^2\)$ are nonpower subgroups of $G$. Similarly, the five 
			conjugates of $\(a\)K$ and the five conjugates of $\(a^2\)K$ are nonpower subgroups 
			of $G$. Since $\nps(G)\le 13$ we must have $K\subseteq \(a\)$ and therefore $P$ is 
			cyclic. Therefore $G\simeq G^{(2)}_{2,n;\,5,1}$.\par}
		
		\Case{1c}{$p = 2$ and $|Q| = 7$.}
		{\narrow
			In this case $G/K\simeq D_{14}$ and $P$ has 7 conjugates. Suppose that
			$a\in P\setminus K$. Then $P$, $\(a\)$ and their conjugates are nonpower subgroups.
			But $\nps(G)\le 13$ and so $P = \(a\)$. Therefore $G\simeq G_{n,7}$ whence
			$\nps(G) = 7$ and there are no examples with $10\le\nps(G)$.\par}
		
		\Case{1d}{$p = 2$ and $|Q| = 9$.}
		{\narrow
			Suppose that there is a subgroup $R$ of order 3 in $Q$ such that $R\normal G$.
			Then $RP$ is not normal in $G$, otherwise the Frattini argument 
			\cite[Th.\ 1.3.7]{gorenstein:1968} implies $G = RN_G(P) = RP$, which is a 
			contradiction. The 9 conjugates of $P$ and the 3 conjugates of $RP$ are nonpower 
			subgroups. Therefore all other subgroups are normal. Thus $R$ is the unique 
			subgroup of $Q$ of order 3 and hence $Q$ is cyclic. If $a\in P\setminus K$ and 
			$b$ generates $Q$, then $b^a = b^{-1}$. If $P\ne \(a\)$, then $\(a\)\normal G$ 
			and so $ab = ba$, which is a contradiction.  Therefore $P$ is cyclic, 
			$\nps(G) = 12$ and $G\simeq G_{n,9}$.
			
			This leaves us with the possibility that no subgroup of order 3 is normal in $G$.
			This implies $Q$ is not cyclic, otherwise the subgroup of order three in $Q$ would 
			be a normal subgroup of $G$. Consequently
			$Q\simeq C_3\times C_3$ and its four subgroups of order 3 are not normal in $G$.
			Since $\nps(G)\le 13$ and the 9 conjugates of $P$ are nonpower subgroups, it
			follows that the proper subgroups of $P$ are normal in $G$. As above, this implies
			that $P$ is cyclic.  Furthermore, if $P = \(a\)$, then $\(a^2\)\normal G$ and so
			$a^2\in K$. Thus if $x$ is an element of order 3 in $Q$ and $y = x^a\ne x$, then
			$(xy)^a = xy$ and $\(xy\)\normal G$, which is a contradiction.\par}
		
		\Case{1e}{$p = 2$ and $|Q| = 11$.}
		{\narrow
			In this case $G/K\simeq D_{22}$ and $P$ has 11 conjugates.
			Thus $P$ is cyclic and $G\simeq G_{n,11}$.\par}
		
		\Case{1f}{$p = 2$ and $|Q| = 13$.}
		{\narrow
			In this case $G/K\simeq D_{26}$ and $P$ has 13 conjugates.
			Thus $P$ is cyclic and $G\simeq G_{n,13}$.\par}
		
		\Case{2}{$p = 3$ and $m\in\{4,7,13\}$.}
		In this case $G = Q\rtimes P$, where $P$ is a 3-group and $Q$ is a $q$-group with
		$q\in \{2,7,13\}$.
		\Case{2a}{$p = 3$ and $|Q| = 4$.}
		{\narrow
			Considering the homomorphism $G\to \Sym(4)$ (with kernel $K\subset P$), we have 
			$G/K \simeq \Alt(4)$. Thus $Q\simeq C_{2}\times C_{2}$. From $\nps(\Alt(4)) = 7$ 
			and $10\le\nps(G)$ it follows that $K\ne 1$. Then $G$ has 7 nonpower subgroups 
			that properly contain $K$ and 3 nonpower subgroups that are non-trivial proper 
			subgroups of $Q$. By $\nps(G)\leq 13$, we have $\nps(G/Q)\leq 3$. This implies 
			$P\simeq C_{3^n}$ and hence $G\simeq A_{n}$. From Lemma \ref{lemma:list} we have 
			$G\simeq A_2$ or $A_3$.\par}
		
		\Case{2b}{$p = 3$ and $|Q| = 7$.}
		{\narrow
			From Theorems \ref{thm:upto4} and \ref{thm:upto9}, Lemma \ref{lemma:quo} and the
			assumption $\nps(G)\le 13$, $P$ is either cyclic or isomorphic to $C_3\times C_3$.
			If $P\simeq C_{3} \times C_{3}$ and $H$ is a non-central subgroup of order 3, then 
			$H$ has 7 conjugates. This contradicts $\nps(G)\le 13$ because $P$ has three 
			subgroups of order 3 not in the centre of $G$. Thus $P$ is cyclic. Let $a$ be a 
			generator of $P$. Then there exists a generator $b$ of $Q$ such that $a^{-1}ba = b^2$. 
			Therefore, for some $n\ge 1$, $G\simeq F_{n,7}$. But then $\nps(G) = 7$ whereas we
			assume that $10\le \nps(G)$.\par}
		
		\Case{2c}{$p = 3$ and $|Q| = 13$.}
		{\narrow
			As in previous cases $P$ is cyclic and thus $G\simeq F_{n,13}$.\par}
		
		\Case{3}{$p = 5$ and $|Q| = 11$.}
		As in previous cases $P$ is cyclic and thus $G\simeq G_{5,n;\,11,1}^{(3)}$.
		
		\Case{4}{$p = 7$ and $|Q| = 8$.}
		The subgroup $P$ acts transitively on the 7 subgroups of order 2 of $Q$ and on 
		the 7 subgroups of order 4 otherwise $Q\subseteq C_G(P)$. Thus $\nps(G)\ge 14$ 
		and so there are no examples with $p = 7$.
	\end{proof}
	
	\begin{proof}[Proof of Theorem \ref{thm:main}]
		From above discussion we may suppose that $G$ is nilpotent so that $G = P\times H$ 
		where $P$ is a non-cyclic $p$-subgroup and $H\ne 1$ is a group whose order is not 
		divisible by $p$. If $H$ is also non-cyclic, then $H$ has at least 3 subgroups and
		from Lemma \ref{lemma:prod} we have $\nps(G)\ge 15$. Thus $H$ is cyclic and 
		$\nps(P\times H) = \nps(P)s(H)$. 
		
		If $H\ne 1$, then $s(H)\ge 2$ and therefore the assumption $\nps(G)\le13$ implies 
		$\nps(P)\le 6$. Referring to Theorems \ref{thm:upto4} and \ref{thm:upto9}, $P$ is 
		isomorphic to $C_2\times C_2$, $C_3\times C_3$, $C_2\times C_4$ or $C_5\times C_5$.
		If $\nps(G) = 10$, then $G\simeq C_2\times C_4\times C_p$, where $p > 3$ is a prime;
		if $\nps(G) = 12$, then $G\simeq Q_8\times C_{qr}$, $C_2\times C_2\times C_{qr}$,
		$C_3\times C_3\times C_{s^2}$ or $C_5\times C_5\times C_p$, where $p,q,r$ and $s$ 
		are primes such that $p\ne 5$, $3 < q < r$ and $s\ne 3$.
		
		From now on we may suppose that $G$ is a non-cyclic $p$-group. Then $G/\Phi(G)$ is an 
		elementary abelian group of order $p^d$ and every proper non-trivial subgroup of
		$G/\Phi(G)$ is a nonpower subgroup. Since $\nps(C_p\times C_p\times C_p)\ge 
		\nps(C_2\times C_2\times C_2) = 14$ it follows that $d = 2$ and $G$ is generated by 
		two elements. Thus from Lemma \ref{lemma:list}\,\eqref{i:Cp} we have 
		$G/G'\simeq C_2\times C_{2^n}$ for $1\le n\le 6$, 
		$C_3\times C_{3^n}$ for $1\le n\le 4$, $C_5\times C_{5^n}$ for $n\in\{1,2\}$, 
		$C_7\times C_7$, $C_{11}\times C_{11}$ or $C_4\times C_4$. If $G$ is abelian, the 
		proof is complete. From now on we assume that $G'\ne 1$.
		
		If $G/G'= C_2\times C_2$ it follows from \cite[Th.\ 5.4.5]{gorenstein:1968} that
		$G$ is is isomorphic to $D_{2^m}$, $S_{2^m}$ or $Q_{2^m}$ and then from 
		Lemma \ref{lemma:list} the only possibility is $S_{16}$.
		
		Since $G'\ne 1$ there exists $R\normal G$ such that $|G'/R| = p$. We shall determine
		the structure of $G/R$ for for each choice of $G/G'$.
		
		If $G/G' = C_p\times C_p$ and $p$ is odd, it follows from 
		\cite[Th.\ 5.5.1]{gorenstein:1968} that $G/R$ is an extraspecial group of order
		$p^3$ and then from Lemma \ref{lemma:list} the only possibility for $G/R$ is 
		$M_{3,5}$. We argue as in \cite{zheng-etal:2023}. The group $M_{3,5}$ has a cyclic 
		subgroup of order 25 and therefore it is metacyclic. It follows from Lemma \ref{lemma:R} 
		that $G$ is metacyclic and so $G$ has a cyclic normal subgroup that properly 
		contains $G'$; that is, $G$ has a cyclic subgroup of index 5 and therefore, by 
		\cite[Th.\ 5.4.4]{gorenstein:1968} we have $G\simeq M_{n,5}$.  From
		Lemma \ref{lemma:list} it follows that $n = 3$ and $G\simeq M_{3,5}$.
		
		If $G/G'\simeq C_2\times C_{2^6}$, $C_3\times C_{3^4}$, $C_{11}\times C_{11}$ or 
		$C_4\times C_4$, then $\nps(G/G')\in\{12,13\}$ and  all subgroups of $G$ 
		that do not contain $G'$ are normal. But then every subgroup of $G$ is normal because
		all subgroups that contain $G'$ are normal. Thus $G$ is a Hamiltonian group whence $p = 2$ 
		and $G\simeq Q_8\times C_2\times\dots\times C_2$ (see \cite[III, 7.12]{huppert:1967}).
		It follows from Lemmas \ref{lemma:list} and \ref{lemma:hamilton} that no groups of  this
		type satisfy $10\le\nps(G)\le 13$. Similarly, because $\nps(C_5\times C_{5^2}) = 11$, 
		the same argument shows that there are no examples with $G/G' \simeq C_5\times C_{5^2}$.
		
		Finally we consider the cases $G/G'\simeq C_2\times C_{2^n}$ for $2\le n\le 5$
		and $C_3\times C_{3^{n}}$ for $n = 2,3$. If $G/R$ has a cyclic maximal subgroup, 
		it follows from \cite[Th.\ 5.4.4]{gorenstein:1968} that $G\simeq M_{n+2,2}$ or
		$M_{n+2,3}$. From Lemmas \ref{lemma:qd} and \ref{lemma:list}, $G$ is isomorphic to
		$M_{6,2}$, $M_{7,2}$, $M_{4,3}$ or $M_{5,3}$. If the exponent of $G/R$ is $2^n$ or 
		$3^n$, then from Lemmas \ref{lemma:R} and \ref{lemma:B}, $G/R\simeq B^2_{2,2}$
		and $\nps(G/R) = 12$.
		
		We may choose generators $a$ and $b$ for $G$ such that their images $\overline a$ 
		and $\overline b$ in $G/R$ satisfy the relations $\overline a^4 = 1$, 
		$\overline b^4 = 1$ and ${\overline a}^{-1}\overline b\overline a = {\overline b}^{-1}$.
		The nonpower subgroup $\(\overline a\)$ is the image in $G/R$ of both $\(a\)$ and 
		$\(a\)R$. Since $\(a\)$ is not normal in $G$ we must have $R\subseteq \(a\)$
		and therefore $R = \(a^4\)$ is cyclic. The group $G$ has at most one power subgroup
		of order 2 and since $G$ is neither cyclic nor a generalised quaternion group it has 
		at least three subgroups of order 2.  Therefore there are elements $x,y\in G\setminus R$
		of order 2 such that $\(x\)$ and $\(y\)$ are nonpower subgroups.  But then
		$\(x\)R$ and $\(y\)R$ are also nonpower subgroups and this is a contradiction.
		Therefore, $R = 1$, $G\simeq B^2_{2,2}$ and this completes the proof.
	\end{proof}
	
	\begin{rem}
		All groups that occur in Theorems \ref{thm:upto4}, \ref{thm:upto9} and 
		\ref{thm:main} have the property that for at most one prime $p$ the group 
		has more than one Sylow $p$-subgroup.  The smallest example of a group that
		does not have this property is the semidirect product $C_7\rtimes C_6$ (with
		trivial centre); it has 21 nonpower subgroups.
	\end{rem}

  \noindent\textbf{Acknowlegments.} The authors are grateful to the referee for her/his valuable
  comments and for the careful reading of this paper.
	


\begin{thebibliography}{99}
		
		\bibitem{anabanti-etal:2022}
		C. S. Anabanti,  A. B. Aroh, S.~B. Hart, A.~R. Oodo, A question of {Z}hou, {S}hi and {D}uan on nonpower subgroups
		of finite groups, \emph{Quaest. Math} \textbf{45}(6)(2022), 901-910. 
		
		\bibitem{anabanti-hart:2022}
		C.~S. Anabanti, S.~B. Hart, Groups with a given number of nonpower subgroups, \emph{Bull. Aust. Math. Soc} \textbf{106}(2)(2022),  315-319.
		
		\bibitem{berkovich:2008}
		Y. Berkovich, \emph{Groups of prime power order, {V}ol. 1}( Walter de Gruyter GmbH \& Co. KG, Berlin, 2008).
		
		\bibitem{blackburn:1958}
		N. Blackburn, On prime-power groups with two generators, \emph{Proc. Cambridge Philos. Soc.} \textbf{54}(1958), 327-337.
		
		\bibitem{burnside:1955}
		W. Burnside,  \emph{Theory of groups of finite order}, 2nd edn.( Dover Publications, Inc., New York, 1955).
		
		\bibitem{gorenstein:1968}
		D. Gorenstein,  \emph{Finite Groups}( Harper \& Row, New York, 1968).
		
		\bibitem{hall:1934}
		P. Hall, A {C}ontribution to the {T}heory of {G}roups of
		{P}rime-{P}ower {O}rder. \emph{Proc. London Math. Soc.} \textbf{36}(2)(1934), 29-95.
		
		\bibitem{hempel:2000}
		 C. E. Hempel,  Metacyclic groups, \emph{Comm. Algebra} \textbf{28}(8)(2000), 3865-3897.
		
		\bibitem{huppert:1967}
		B. Huppert, \emph{Endliche {G}ruppen. {I}}( Springer-Verlag, Berlin, New York).
		
		\bibitem{mann:2010}
	    A. Mann, The number of subgroups of metacyclic groups, \emph{ Character theory of finite groups, Contemporary Mathematics} \textbf{524} (2010), 93-95.
		
		\bibitem{miller:1929}
		G. A. Miller,  Number of {A}belian {S}ubgroups in {E}very {P}rime {P}ower {G}roup, \emph{Amer. J. Math.} \textbf{51}(1)(1929), 31-34.
		
		\bibitem{szasz:1956}
		F. Sz\'{a}sz, On cyclic groups, \emph{Fund. Math.} \textbf{43}(1956), 238-240.
		
		\bibitem{Marius:2010}
		M. T\u{a}rn\u{a}uceanu, An arithmetic method of counting the subgroups of a finite
		abelian group, \emph{Bull. Math. Soc. Sci. Math. Roumanie (N.S.)}(2010), 373-386.
		
		\bibitem{tarnauceanu-toth:2017}
		M. T\u{a}rn\u{a}uceanu, and L. T\u{o}th, On the number of subgroups of a given exponent in a finite abelian group, preprint(2015), https://arxiv.org/abs/1507.00532.
		
		\bibitem{zheng-etal:2023}
		J.~W. Zheng, W. Zhou, D.E.Taylor,  Groups with few nonpower subgroups. \emph{Bull. Aust. Math. Soc.} (2023), 1-12.
		
		\bibitem{zhou-shi-duan:2006}
		W. Zhou, W. J. Shi, Z. Y. Duan, A new criterion for finite noncyclic groups. \emph{Comm. Algebra} \textbf{34}(12)(2006),  4453-4457. 
		
		
	\end{thebibliography}
\end{document}